\documentclass[12pt]{article}
\usepackage{amssymb, amsmath,amsthm,amsfonts,a4,bezier}
\usepackage{hyperref}
\usepackage{amscd}
\usepackage{graphpap}
\usepackage[pdftex]{color,graphicx}
\usepackage{color,graphicx}
\usepackage{mathpazo}
\usepackage{setspace}
\usepackage{comment}
\usepackage[draft]{todonotes}

\DeclareGraphicsExtensions{.jpg,.pdf,.eps,.png}
\newtheorem{theo}{Theorem}

\newtheorem{Definition}{Definition}

\newtheorem{prop}{Proposition}

\newtheorem{cor}{Corollary}

\newtheorem{Remark}{Remark}

\newcommand{\Z}{{\mathbb {Z}}}

\begin{document}
\thispagestyle{empty}
\baselineskip=28pt
\vskip 4mm
\begin{center} {\LARGE{\bf On a Local Version of the  Bak-Sneppen Model }}
\end{center}

\baselineskip=12pt
\vskip 10mm

\begin{center}\large
Iddo Ben-Ari\footnote{Department of Mathematics, University of Connecticut, USA, \href{mailto:iddo.ben-ari@uconn.edu}{iddo.ben-ari@uconn.edu}} 
 and Roger W. C. Silva 
 \footnote{Departamento de Estat\'{\i}stica,
 Universidade Federal de Minas Gerais, Brazil, \href{mailto:rogerwcs@ufmg.br} {rogerwcs@ufmg.br}}
\\ 

\baselineskip=15pt
\vskip 10mm
\vskip 10mm

\end{center}

\begin{abstract} A major difficulty in studying the Bak-Sneppen model is in effectively comparing it with well-understood models. This stems from the use of two geometries: complete graph geometry to locate the global fitness minimizer, and graph geometry to replace the species in the neighborhood of the minimizer.  We present a variant in which  only the graph geometry is used. This allows to  obtain the stationary distribution through random walk dynamics. We use this to show that for constant-degree graphs, the  stationary fitness distribution converges to an IID law as the number of vertices tends to infinity. We also discuss exponential ergodicity through coupling, and avalanches for the model.  

~\\
\noindent{\it Keywords: Bak-Sneppen; species; fitness; stationary distribution} 

\noindent {\it AMS 1991 subject classification:60K35; 60J05; 92D15} 
\end{abstract}

\onehalfspacing
\section{Introduction}
\subsection{Background}
\label{sec:bg}
The Bak-Sneppen model was introduced in \cite{BS} as a toy model for biological evolution. Let $G=(V,E)$ be a finite connected undirected graph.  For $u,v \in V$, we write $u\sim v$ if $\{u,v\}\in E$ or if $u=v$.  The model is a discrete time Markov chain on the state space ${\cal S}$  of nonnegative  one-to-one functions on the vertex set $V$.  A state $\Gamma$ is to be interpreted as a collection of fitness values $(\Gamma(v),v\in V)$ assigned to the species at the vertices of $G$.  The rule of the transition is very simple. From state $\Gamma$, the system samples the next state $\Gamma'$ according to the following rule:
\begin{enumerate} 
\item[A-1.]\label{it:glob}  Let $v=\mbox{argmin }\Gamma$. 
\item[A-2.] $\Gamma'(w) = \Gamma(w)$ if $w\not\sim v$.  
\item[A-3.] $(\Gamma'(w):w\sim v)$ are IID $ \mbox{Exp}(1)$ independent of past. 
\end{enumerate} 
In other words, species with lowest fitness is removed from the system, along with its graph neighbors, and those are replaced by new species with IID $\mbox{Exp}(1)$  fitnesses instead.  Note that the dynamics of the process guarantee that (with probability $1$), a species can be identified with its fitness. The model is a Markov chain, and if $G$ is connected, it is also exponentially ergodic. In the original model $G$ was the $N$-cycle, $V=\{0,\dots,N-1\}$ and $E=\{\{n,n+1 \mod N\}:n\in V\}$. Also, in the original model, the fitnesses assigned were $U[0,1]$  rather than $\mbox{Exp}(1)$, which, as is easy to see, has no effect on the dynamics. Furthermore, a  deterministic transformation maps one model  in a  one-to-one  fashion into the other (e.g. map fitness $b$ in $\mbox{Exp}(1)$-model to $1-e^{-b}$ to obtain the $U[0,1]$-model). We choose to work with exponential distribution, even when referring to the original model. Numerical simulations for the original model (see \cite{BS} and \cite{J} for instance) show that the stationary distribution converges weakly  to a product law as $N\to\infty$, where the marginals are $\mbox{Exp}(1)$, conditioned to be above a certain threshold $b_c$ (or, equivalently, exponentials shifted by that threshold), with $e^{-b_c}$ estimated around $\frac 13$. This is an open problem. The physics literature contains other interesting and concise predictions based on numerical simulations which will not be discussed here. 

Mathematically rigorous results on the model were obtained by Ronald Meester and his collaborators in a sequence of papers introducing the graphical representation (hence the $\mbox{Exp}(1)$-model, interpreting  fitnesses as times of events of rate-$1$ Poisson processes) as the main tool to analyze the model and  specifically  the regenerative structure of  the model known as avalanches. In \cite{MZ} the authors establish non-triviality for the original model, in the sense that the marginal fitness distribution under the stationary distribution does not converge to a point mass at $+\infty$ as $N\to\infty$. This result highlights the role of the graph geometry through ``local interaction". If geometry is ignored by considering   $E=\emptyset$, then for each $N$, the fitness at each site converges to $+\infty$ a.s. as time tends to $+\infty$. Avalanches for the Bak-Sneppen model are defined as follows. Given a fitness threshold $b$, an avalanche from  this threshold is any part of the path from a time where the minimal fitness is at least $b$ until the next time this occurs. The important feature of avalanches is that their evolution depends on the past only through the vertex with the minimal fitness at the time the avalanche begins. This is because from that point onwards, the avalanche is determined by the new fitnesses sampled, as all previously sampled fitness are at least $b$, and won't play a role in the dynamics until the completion of the avalanche.  In particular, if the graph is vertex-transitive (as in the original model), then  for a fixed avalanche threshold, the duration of an avalanche and the number of sites at which species are replaced  during an avalanche, also known as the range of the avalanche,  are independent of its starting state. This implies that fixing an avalanche threshold and a path of the process, and considering the respective sequences of statistics, that is the sequence of durations of the avalanches along the path, and the corresponding sequence of  ranges, then each of the sequences is IID.  In \cite{MZ2} the authors introduced three  non-trivial  fitness  thresholds for avalanche statistics for the original model, corresponding to behavior as $N\to\infty$. Here is a brief summary. Writing $D_{N,b}$ and $R_{N,b}$ for random variables whose distribution coincides with the duration and range of an avalanche from threshold $b$  for the $N$-cycle, then it was shown that there exist  $0<b_c^D\le b_c^R\le b_c^P<\infty$ such that 
$$ b_c^D = \sup\{b: \lim_{N\to\infty} E[ D_{N,b}]<\infty\},~b_c^R = \sup\{b:\lim_{N\to\infty} E [ R_{N,b}] < \infty],$$
$$~b_c^P = \sup\{b : \lim_{N\to\infty} P [ R_{N,b} = N]=0\},$$ 
and that 
\begin{enumerate}
\item $b_c^D= b_c^R$; and
\item if $b_c^R = b_c^P$ then the stationary distribution for the $N$-cycle converges as $N\to\infty$ to a product measure of $\mbox{Exp}(1)$ conditioned to be above $b_c^P$. 
\end{enumerate}
An important property of the critical avalanche thresholds is that although being defined as limit quantities associated with the $N$-cycle as $N\to\infty$, they are also equal to the corresponding avalanche thresholds on the ``limit" graph $\Z$, when  avalanche on $\Z$ are defined as follows.  Fix a threshold $b\in [0,\infty)$. At time $t=0$, assign fitness $+\infty$ to all vertices in $\Z$, except the origin where the fitness assigned is $b$, then continue the process according to the rules  A-1. and A-2. with $G=\Z$, until the first time all fitnesses are at least $b$. This (possibly infinite) time,  $D_{\infty,b}$,  is the avalanche duration,  and the range, $R_{\infty,b}$,  is the number of vertices whose species were replaced during the avalanche. Then, it was shown that 
$$ \lim_{N\to\infty} E [ D_{N,b} ] = E [D_{\infty,b}],~\lim_{N\to\infty} E [ R_{N,b} ] = E [R_{\infty,b}],$$ and that 
$$\lim_{N\to\infty} P(R_{N,b} = N)= P( R_{\infty,b} = \infty).$$ 
That is, the thresholds $b_c^D,b_c^R,b_c^P$ are intrinsic to the avalanche on $\Z$ defined above. 

A key element in the proof of the second statement is the notion of locking thresholds. It can be shown that if at time $t=0$ the fitnesses are IID $\mbox{Exp}(1)$, then at each time $t\in \Z_+$, the model has an associated (random) function $L_t:V \to [0,\infty)$, such that conditioned on $L_t$: 
\begin{enumerate} 
\item the fitness at a vertex $v$ is $\mbox{Exp}(1)$ conditioned to be above $L_t(v)$; and
\item the fitnesses are independent.
\end{enumerate}
The idea is then to show that equality of $b_c^R$ and $b_c^P$ implies that asymptotically in time and in $N,$ the distribution of the locking thresholds converges to a product of point masses at $b_c^P$.

One of the key difficulties in studying the Bak-Sneppen model, particularly the original model on the $N$-cycle, is the difficulty in comparing and  coupling it with classical well-understood models. In \cite{GMN} the authors were able to obtain two-sided bounds on the analogs for  $b_c^P$ for vertex-transitive infinite graphs of degree $d<\infty$, where the avalanches are defined analogously to the construction on $\Z$.  This was done by observing  that the range of an avalanche is stochastically dominated by a branching process, which yields a lower bound on $p_c^P$,  and  by constructing a coupling with a site percolation model in which the cluster containing the origin is a subset of  the vertices where species were replaced during the avalanche, giving an upper bound on $b_c^P$. The bounds are asymptotically sharp for $\Z^d$ and $d$ regular trees as $d\to\infty$.

We note that the stationary distribution for the original Bak-Sneppen model was computed for $N=4$ and  $N=5$ in \cite{S1} and \cite{S2}, respectively, as well as in \cite{ALS}. In the latter work, the  authors obtained a formula for the stationary distribution for a variant of the Bak-Sneppen in which at each step all species are replaced except for one (coinciding with original Bak-Sneppen if $N=4$) or two (coinciding with original Bak-Sneppen if  $N=5$). 

Over the years, several additional variations of the Bak-Sneppen model were studied in the mathematics literature.  A discrete-fitness variant on the $N$-cycle was introduced in \cite{BK}. In this process, a configuration is a binary word with $N$-bits, and at each step one chooses a random bit of minimum value and replaces it and its two neighbors by independent Bernoulli random variables with parameter $p\in(0,1)$. The authors obtain bounds on the average number of ones in the stationary distribution as $N\to\infty$. In \cite{MS} the authors consider a model where, at each discrete time step, the particle with the lowest fitness is chosen, together with one other particle, chosen uniformly at random among all remaining particles. It is shown that this model obeys power law behavior for avalanche duration and size, behavior observed in numerical simulations for the original Bak-Sneppen model.  Another variant of the Bak-Sneppen model was introduced in  \cite{GMS}, studied in \cite{BMR}, and generalized in  \cite{MV}.  Here the number of species follows a path of reflected random walk transient to infinity. When the population increases, new species with IID $\mbox{U}[0,1]$-fitness are added, and when the population decreases, the species with lowest fitnesses are removed from the population. For this model, the empirical fitness distribution tends to a uniform law on $[p_c,1]$, similar to observed empirical  behavior of Bak-Sneppen  on the cycle as $N\to\infty$. Note that the geometric aspect of the Bak-Sneppen model is lost in this model, since the population has no spatial structure. A scaling limit for the empirical distribution for this model was obtained in \cite{B}.  
\subsection{Local Bak-Sneppen}
In this paper  we present yet another ``relative" of the Bak-Sneppen model.  Indirectly, this will shed some light on the Bak-Sneppen model. More directly, this is part of an effort to  find a family of solvable models which can be coupled with or effectively compared with the Bak-Sneppen model. 

To further motivate our work, we present a class of models which contains both the Bak-Sneppen model as well as our new model, and which stresses the role of geometry in the dynamics. Recall our finite  undirected  graph $G=(V,E)$. Construct a second undirected  graph $G'=(V,E')$. As before, for $u,v\in V$, write $u\sim' v$ if $\{u,v\}\in E'$ or if $u=v$. We consider a new Markov chain on state space $V\times {\cal S}$, where ${\cal S}$ is the space of nonnegative one-to-one functions on the vertex set $V$, by slightly tweaking the dynamics of the Bak-Sneppen model as follows. From state $(u,\Gamma)$ the next state $(v,\Gamma')$ is sampled according to the rules presented above, but with A-1 replaced by 
\begin{enumerate} 
\item[A-1'.] $v= \mbox{argmin }\{  \Gamma(w):w\sim' u \}$, 
\end{enumerate} 
and keeping A-2, A-3.  In other words, instead of looking for the global minimum, we look for the local minimum in the $G'$-neighborhood of $u$. Or: we look for minima in the $G'$-neighborhood, and update fitnesses in the $G$-neighborhood.  The  Bak-Sneppen model corresponds to  $G'$ being the complete graph, in which case there is no need to track the local-minimum: it coincides with the global minimum which is a function of the state $\Gamma$. 

We will study the case where there is only one geometry, namely  $G'=G$, and call it the local Bak-Sneppen model.  In the sequel, we will always assume that $G$ is connected. Our main result provides a complete description of the stationary  measure for the local Bak-Sneppen model. 

We now highlight some of the similarities and differences between the Bak-Sneppen and the local Bak-Sneppen model, as seen through our results. First and foremost, a general or asymptotic formula for the stationary distribution for the  Bak-Sneppen model is an open problem. One of the main results in this paper (Theorem \ref{th:main}) is a formula for the local Bak-Sneppen model. Our formula does provide asymptotic IID structure for the fitnesses for $d$-regular graphs as the number of vertices tends to infinity (Proposition \ref{limit_dist}), a feature which is expected to hold for the original  Bak-Sneppen model.  The limiting fitness distribution in our model is exponential, conditioned not to be the minimum among  $(d+1)$-IID $\mbox{Exp(1)}$ random variables, or $\mbox{Exp}(1)$ conditioned to be above a random threshold. This differs from the expected expression for  Bak-Sneppen,  which is $\mbox{Exp}(1)$ conditioned to be above a deterministic threshold. This difference has a simple heuristic explanation due to the difference in the geometry:  in the original Bak-Sneppen fitnesses at all vertices are considered at every step when looking for the global  (complete graph) minimum, effectively eliminating all small values, while in our case, once assigned, fitness is  considered  at most once (if at a site in the neighborhood of the local minimum) before being replaced, a mechanism that cannot exclude very small fitnesses. 

\section{Results} 
We will assume that $G=(V,E)$ is a finite, undirected and connected graph. 
\subsection{Description of the Model}\label{sec:erg}
Let ${\bar \Omega}=  V \times [0,\infty)^V$, equipped with the   Borel $\sigma$-algebra induced by the product of the discrete topology on $V$ and the Euclidean metric on each of the $|V|$ copies of  $[0,\infty)$.  We write $\Omega \ni \omega = (v_\omega,\Gamma_\omega)$, where $v_\omega \in V$ and $\Gamma_\omega: V \to [0,\infty)$, and let $\Omega\subset {\bar \Omega}$ denote the set consisting of all elements $\Omega\ni \omega = (v_\omega,\Gamma_\omega)$ with $\Gamma_\omega$ one-to-one, and equip it with the  Borel $\sigma$-algebra, induced by the subspace topology. The set $\Omega$ will serve as the state space for our process, and the restriction on $\Gamma_\omega$ is to be understood as distinct fitnesses at distinct vertices, a requirement needed for the dynamics to be well-defined. Also, for $v\in V$, we let $A_v = \{u \in V : u \sim v\}\cup\{v\}$, that is $A_v$ consists of all neighbors of $v$ in $G$ and $v$ itself. 
 
 The process is denoted by $\Xi=(\Xi_t:t\in\mathbb{Z}_+)$, with $\Xi_t = (X_t,\Gamma_t) \in \Omega$.  Conditioned on   $(\Xi_s:s\le t)$,  $\Xi_{t+1}=(X_{t+1},\Gamma_{t+1})$ is defined as follows:  \begin{equation*}
 \left\{ 
 \begin{array}{l} 
  X_{t+1}=\mbox{argmin}_{u \in A_{X_t}} \Gamma_t (u),\\
   (\Gamma_{t+1}(u):u \in A_{X_{t+1}})\mbox{ are IID Exp}(1),\mbox{ independent of }\Xi_0,\dots,\Xi_t,\mbox{ and }\\
  \Gamma_{t+1}(v) = \Gamma_t (v),~v \not\in A_{X_{t+1}}.
 \end{array}
 \right. 
 \end{equation*}
Clearly, $\Xi$ is a discrete time Markov process. Observe  that it follows directly  from the definition that  the vertex process $X=(X_t:t\in\mathbb{Z}_+)$ is a random walk on $G$, with transition function $p(v,v') = \frac{\mathbb{1}_{A_v}(v')}{|A_v|}$. In particular, $X$ is ergodic with stationary measure $\mu$, given by  
$$\mu(v) = \frac{\deg(v)+1}{\sum_{v\in V} ( \deg (v)+1)}=\frac{|A_v|}{\sum_{u\in V}|A_u|}=\frac{|A_v|}{S_G},$$
where $S_G =\sum_{u \in V}|A_u|$.
Unless clearly specified, whenever a random walk is mentioned in this paper, the reader should have in mind a process with the above transition probabilities.
\subsection{Ergodicity}
\label{sec:ergodicity}
 If $\gamma$ is a distribution on $\Omega$, write $P_\gamma$ for the distribution of $\Xi$ when $\Xi_0$ is $\gamma$-distributed. If $\eta \in \Omega$, we abuse notation and  write $P_\eta$ for $P_{\delta_\eta}$,  where $\delta_\eta$ is the Dirac-delta measure at $\eta$. 

Let $Q_1,Q_2$ be two probability measures on $\Omega$. We define the total variation distance between $Q_1$ and $Q_2$ as 

$$\|Q_1-Q_2\|_{TV}=\sup_A \left |Q_1(A) -Q_2(A)\right|.$$ 
We also write $P_\gamma (\Xi_t \in \cdot)$ for the distribution of $\Xi_t$ under $P_\gamma$. 
We write 
$$\bar{d}_t=\sup_{\eta,\eta'\in \Omega}  \|P_{\eta} (\Xi_t \in \cdot) - P_{\eta'} (\Xi_t \in \cdot) \|_{TV}.$$ 
We have the following exponential ergodicity statement. 
\begin{prop}\label{experg}
	There exist constants $c>0$ and $\beta\in (0,1)$, depending only on $G$, such that for any $t\in {\mathbb Z}_+$, 
	\begin{equation} 
	\label{eq:totval} \bar{d}_t \le c \beta^t.
	\end{equation}
\end{prop}
Before we prove Proposition \ref{experg}, we state a standard corollary whose proof we leave to the Appendix. 
\begin{cor}
\label{cor:unique_stat_dist}
$\Xi$ has a unique stationary measure $\pi$. 
\end{cor}\label{cor:d_bounds}
We note that much of our discussion in the following sections will focus on a detailed  description of the stationary distribution $\pi$, and, in particular, obtaining an explicit formula for it that allows to study the model along a sequence of graphs.

Let
$$d_t = \sup_{\eta } \| P_\eta (\Xi_t \in \cdot) - \pi\|_{TV}.$$ 
Then a standard convexity argument and the triangle inequality imply 
 $$ 
 d_t \le \bar{d}_t \le 2 d_t.
 $$
Here is another  useful yet standard corollary, obtained from the exponential ergodicity and  whose  proof is left to the Appendix. 
\begin{cor}\label{cor:erg}
  		Let    $f:\Omega \to {\mathbb R}$ be bounded and measurable.  Then 
  		$$ \lim_{n\to\infty} \frac{1}{n}\sum_{t=0}^{n-1} f(\Xi_t) = \int f d\pi.$$
  		in $P_\eta$-probability for every $\eta$. 
  \end{cor}  

In order to prove Proposition \ref{experg} we remember several  notions that will be used in the sequel. Recall that if $Y=(Y_{t}:t\in\Z_+)$ and $Y'=(Y_t:t\in\Z_+)$ are both Markov processes defined on a common probability space having the same transition kernel, then the pair $(Y,Y')$ is called a {\bf coupling} for $Y$. Given a coupling, we write 
$$\tau_{coup}(Y,Y')=\inf\{t\in\Z_+:Y_t=Y'_t\}$$ 
for the {\it coupling time}. A coupling is called {\it successful} if $\tau_{coup}(Y,Y')<\infty$ a.s. All couplings we will construct and consider in the sequel will be successful, and, in addition, will be also coalescing, that is 
$$ Y_t=Y'_t,~\mbox{ for all }t \ge \tau_{coup}(Y,Y').$$ 
Finally, suppose that $X$ is a random walk on a finite connected graph $G$. We define the {\it cover time} for $X$ as 
$$\tau_{cover}(X)=\inf\{t\in\Z_+:\{X_0,\dots,X_t\}=V\}.$$   

\begin{proof}[Proof of Proposition \ref{experg}]
	The idea is to construct a coupling and employ Aldous' inequality. The coupling has two stages. In the first stage we are running the two copies independently until the coupling between the respective random walks on $G$ is successful (the fact that the graph is finite and connected and that the probability to stay in each vertex is positive guarantee that the two random walks will eventually meet). This completes the first stage. In the second stage, the two random walks move together by assigning the same fitnesses in the local neighborhoods for both copies. This  guarantees a successful coupling for the local Bak-Sneppen on or before the cover time of the random walk, starting from the vertex  where the first stage ended.  	
	
	Fix two states $\eta$ and $\eta'$. We will construct a coupling $(\Xi,\Xi')$, where each of the processes is a local Bak-Sneppen on $G$, and  $\Xi_0=\eta$ and $\Xi'_0=\eta'$. To this end, for each $t\ge 0$ and $v\in V$, let $K(t,v)(\cdot)$ be a random vector of $\mbox{Exp}(1)$ random variables, indexed by $u\in A_v$. We will assume that all random variables  $K(\cdot,\cdot)(\cdot)$ are independent. Given $\Xi_t=(X_t,\Gamma_t)$ and $\Xi'_t=(X_t',\Gamma_t')$, we continue as follows.
 Let 
	\begin{itemize}
		\item  $X_{t+1}=\mbox{argmin}_{u \in A_{X_t} }\Gamma_t (u)$ and  
		$$\Gamma_{t+1}(u)=\begin{cases} K(t+1,X_{t+1})(u) & u \in A_{X_{t+1}} \\ \Gamma_t(u) & \mbox{otherwise.}\end{cases}$$ 
		\item 
		$X'_{t+1}=\mbox{argmin}_{u \in A_{X'_t} }\Gamma'_t (u)$ and   
		$$\Gamma'_{t+1}(u)=\begin{cases} K(t+1,X'_{t+1})(u) & u \in A_{X'_{t+1}} \\ \Gamma'_t(u) & \mbox{otherwise.}\end{cases}$$ 
	\end{itemize} 
	Then the resulting  process $(\Xi,\Xi')$ is a coupling. Let $P_{\eta,\eta'}$ denote the distribution of $(\Xi,\Xi')$ as constructed above.  Observe that $(X,X')$ is a coupling of the irreducible and aperiodic random walk on $G$, and by construction  up to time $\tau_{coup}(X,X')$, $X$ and $X'$ are independent. This implies that $\tau_{coup}(X,X')<\infty$ a.s. and that $\tau_{coup}(X,X')$  has a geometric tail. From the definition of $\Xi,\Xi'$ it also follows that $\Gamma_{\tau_{coup}(X,X')}(u)=\Gamma'_{\tau_{coup}(X,X')}(u)$ for all $u \in A_{X_{\tau_{coup}(X,X')}}$, and, consequently, that  $\Gamma_{\tau_{coup}(X,X')+t}(u)=\Gamma'_{\tau_{coup}(X,X')+t}(u)$ for all  $u \in \cup_{s\le t } A_{X_{\tau_{coup}(X,X')+s}}$. Let 
	$$\bar \sigma = \inf\{t\in \Z_+: \cup_{s \le t} A_{X_{\tau_{coup}(X,X')+s}}=V\}.$$ 
	It is clear that $\bar{\sigma}$ is stochastically dominated by the cover time of the random walk on $G$ starting from $X_{\tau_{coup}(X,X')}$. As a result of the finiteness of the graph, $\bar\sigma<\infty $ a.s. and it has a geometric tail. From the construction, $\Xi_t = \Xi'_t$ for all $t\ge \tau_{coup}(X,X')+\bar\sigma$.  By Aldous' inequality we have $$  \|P_{\eta} (\Xi_t \in \cdot) - P_{\eta'} (\Xi'_t \in \cdot) \|_{TV} \le P_{\eta,\eta'}(\tau_{coup}(X,X')+\bar{\sigma}>t) \le c \beta^t,~t\ge 0,$$ for some constant $c>0$ and $\beta\in(0,1)$, whose dependence on $\eta$ and $\eta'$ is only through $(X_0,X_0')$. Since the graph is finite, we may choose $c$ and $\beta$ so that the inequality holds for all $\eta,\eta'$.  
\end{proof} 

\begin{Remark}\label{rem1} The choice of an independent coupling is a generic recipe that works for any graph.  It is easy to see from the proof that the independent coupling in the first stage can be replaced by any coupling satisfying all of the following three conditions: 
\begin{enumerate} 
\item Successful coupling for the random walks $X,X'$ for any choice of $(X_0,X'_0)=(u,v),~u,v \in V$. 
\item For every $t\in \Z_+$,  $\Gamma_{\tau_{coup}(X,X')+t}(u) = \Gamma'_{\tau_{coup}(X,X')+t}(u)$ for all $u \in \cup_{s\le t} A_{X_{\tau_{coup}(X,X') +s}}$. 
\item For every $t\in \Z_+$, the conditional distribution of $(X_{\tau_{coup}(X,X')+t+1},X_{\tau_{coup}(X,X')+t+1})$  on  $$((X_s,X'_s):s \le \tau_{coup}(X,X')+t)$$ is a function of $(X_{\tau_{coup}(X,X')+t},X'_{\tau_{coup}(X,X')+t})$. 
\end{enumerate} 
\end{Remark}
We now derive two-sided bounds on ${\bar d}_t$ for a general graph that are somewhat more explicit than the upper bound of Proposition \ref{experg}. The idea is to get an upper bound expressible as the tail of a sum of two independent random variables. The method works for any coupling satisfying the conditions of Remark \ref{rem1}.  The last condition guarantees that conditioned on $X_{\tau_{coup}(X,X')}$ (or $X'_{\tau_{coup}(X,X')}$,  which is the same), $\tau_{coup}(X,X')$ and $\bar \sigma$ are independent. Nevertheless, if the graph is not vertex transitive, 
there is no guarantee that $X_{\tau_{coup}(X,X')}$ and  ${\bar \sigma}$  are independent. We wish to eliminate this in order to obtain an easier-to-handle upper bound, as the tail of the sum of two independent random variables. The idea is not to modify the coupling, but instead to replace ${\bar \sigma}$ with a random variable which stochastically dominates ${\bar \sigma}$ conditioned on $X_{\tau_{coup}(X,X')}$, whatever the latter may be. To accomplish that, consider a random walk $X$ on $G$ starting from $v$, i.e., $X_0=v$. Let  
\begin{equation*}\label{tilde:sigma}{\tilde\sigma}_v = \inf\{t\ge 0: \cup_{s\le t}  A_{X_s}=V\},
\end{equation*}
and denote by $F_v$ the distribution of ${\tilde\sigma}_v$ and  $F^{\max} =\min_{v\in V} F_v$. Let $\sigma^{max}$ be a random variable with distribution function $F^{\max}$, independent of $\tau_{coup}(X,X')$.  Then $\sigma^{max}$ stochastically dominates ${\tilde\sigma}_v$ for every choice of $v$, and we obtain the following upper bound:
\begin{equation}
\label{eq:nice_upper} {\bar d}_t \le \sup_{u,v} P(\tau_{coup}(X,X') + \sigma^{\max}  >t|(X_0,X'_0)=(u,v)).
\end{equation}

We turn to a lower bound. Let $x\in V$ and suppose now that $\eta=\Xi_0= (x,\Gamma_0),\eta'=\Xi'_0= (x,\Gamma'_0)\in \Omega$,  with  $\Gamma_0(v)= \Gamma'_0(v)$ for all $v\in A_{x}$ and  $\Gamma_0(v) \ne \Gamma'_0(v)$ for all remaining vertices $v\in V-A_{x}$. 
We will make a more specific choice of $\eta,\eta'$. Let  $\epsilon >0$ and $\delta\in (0,1)$.  We will choose our initial state to satisfy $\Gamma_0(v) >1$  for all $v$, and  $ \Gamma'_0(v) < \delta$ for all $v \in V-A_{x}$. Let $C_{\delta,t}$ denote the event that all fitnesses sampled in the coupling  up to time $t$ are $\ge \delta$. By choosing $\delta$ sufficiently small, we can guarantee that $P( C_{\delta,t})>1-\epsilon$. With these initial states,  our coupling gives  $\tau_{coup}(X,X')=0$, hence  $\tau_{coup}(\Xi,\Xi')=\bar{\sigma}$. Therefore it follows that on the event $C_{\delta,t}\cap \{\tau_{coup}(\Xi,\Xi')>t\}$ all fitness in $\Gamma_t$ are $\ge \delta$, and there exists at least one $v\in V$ such that $\Gamma'_t(v)<\delta$. Now let $B_\delta$ be the event that all fitnesses  are $\ge \delta$. Clearly, 
\begin{align*} \bar{d}_t  &\ge  ( P_\eta (\Xi_t \in B_\delta) - P_\eta (\Xi'_t \in B_\delta) )\\
& = E_{\eta,\eta'} [{\bf 1}_{B_\delta}(\Xi_t) - {\bf 1}_{B_\delta} (\Xi'_t),\tau_{coup}(\Xi,\Xi')>t] \\
& \ge E_{\eta,\eta'}  [{\bf 1}_{B_\delta}(\Xi_t) - {\bf 1}_{B_\delta} (\Xi'_t),C_{\delta,t}\cap \tau_{coup}(\Xi,\Xi')>t]-P( C_{\delta,t}^c) \\
& = P_{\eta,\eta'}(C_{\delta,t}\cap \tau_{coup}(\Xi,\Xi')>t) -P( C_{\delta,t}^c) \\
& \ge P( \tau_{coup}(\Xi,\Xi')>t) -2P( C_{\delta,t}^c)\\
& \ge P_{\eta,\eta'}(\bar{\sigma}>t) -2\epsilon=(*)
\end{align*} 
Let $F^{min}=\max_{v\in V} F_v$ and $\sigma^{\min}$  a random variable with distribution function $F^{\min}$. Then, since the distribution of $\bar{\sigma}$ depends only on $x$, $\bar \sigma$ dominates $\sigma^{\min}$ for every choice of $x$, and therefore 
$$(*) \ge P(\sigma^{\min}>t)-2\epsilon.$$ 
Since $\epsilon$ is arbitrary, we conclude that
\begin{equation}
\label{eq:nice_lower}  {\bar d}_t \ge P(\sigma^{\min} > t).
\end{equation}

Clearly, the upper and the lower bounds in Expressions \eqref{eq:nice_upper} and \eqref{eq:nice_lower} are not tight. When $G$ is vertex transitive, then $F^{min}=F^{max}$, and if, in addition, we know that the exponential tail of $\tau_{coup}(X,X')$ is lighter than that of  $\sigma^{min}$, then it is easy to see that the bounds decay at the same exponential tail (a quick way to see this is through the moment generating function of $\sigma^{min}$ and of $\tau_{coup}(X,X')+\sigma^{max}$ which blow up at the same value).  

\subsubsection{Coupling for the $N$-cycle} 
We will improve the lower bound for the case where $G$ is the $N$-cycle. We will assume $N\in 4 {\mathbb N}$. Let  $\eta = (0,\Gamma_0)$ and $\eta'=(N/2,\Gamma_0')$ with $\Gamma_0(v) >\delta^{-1}$ for all $v$, and $\Gamma_0'(v) < \delta$ for all $v$, for some $\delta>0$. For $v \in V$, let ${\bar v}$ denote its ``reflection" about the ``equator" $\{N/4,3N/4\}$: 
$\bar v = (N/2 - v) \mod N$. We will further assume that the minima of  $\Gamma'_0({\bar v})$ and $\Gamma_0(v)$ over $v \in A_{0}$ are attained at the same vertex. Unlike the generic coupling for $X$ and $X'$ we have used for a general graph, here we will consider a reflection coupling, which satisfies the conditions on Remark \ref{rem1}. Specifically, assuming that $(\Xi_s,\Xi'_s)$ is defined for all $s\le t$, we define $(\Xi_{t+1},\Xi'_{t+1})$ inductively. Our induction hypothesis is that if $X_t'\ne X_t$, then $X_t'={\bar X_t}$ and $\Gamma'_t({\bar u}) = \Gamma_t(u)$ for $u \in A_{X_t}$.  Let $X_{t+1}$ be the vertex minimizing $\Gamma_t (u)$ among the three vertices in $A_{X_t}$.  Sample three IID $\mbox{Exp}(1)$ random variables  $U_{-1},U_0$ and $U_1$, independent of the  the past, and let  $\Gamma_{t+1} (X_{t+1} +e \mod N) = U_e,~e \in \{-1,0,1\}$. 
\begin{enumerate}
\item If $X'_t\ne X_t$, then by induction hypothesis,  $X'_t={\bar X_t}$,  and  the minima of $\Gamma'_t({\bar u})$ and $\Gamma_t (u)$ over $u \in A_{X_t}$ are attained at the same vertex. We then let  $X'_{t+1}$ be its reflection about the equator, that is $X'_{t+1} ={\bar X_{t+1}}$. As for $\Gamma'_{t+1}$: 
\begin{enumerate} 
\item If $X'_{t+1}\ne X_{t+1}$, let  $\Gamma'_{t+1} ({\bar v}) = \Gamma_{t+1} (v)$ for $v\in A_{X_{t+1}}$; otherwise 
\item Let $\Gamma'_{t+1}(v) = \Gamma_{t+1} (v)$ for all $v\in A_{X_{t+1}}$. 
\end{enumerate} 
\item If $X_t'=X_t$, let $X'_{t+1}=X_{t+1}$ and let $\Gamma'_{t+1}(u) = \Gamma_{t+1}(u)$ for all $u \in A_{X_{t+1}}$. 
\end{enumerate} 

With this coupling, we have that $\tau_{coup}(X,X')$ is the time $X$ exits the upper semi-circle $\{3N/4+1,\dots, N-1,0,\dots,N/4-1\}$. 
Let $B_X$ denote the vertices $u$ such that $\Gamma_{\tau_{coup}(X,X')}(u)=\Gamma_0(u)$. Then by construction, $\Gamma_{\tau_{coup}(X,X')}(u)\ne \Gamma'_{\tau_{coup}(X,X')}(u)$ for all $u\in B_X$, a.s. and, $B_X$ a.s. contains all elements in the lower semicircle, $\{N/2+1,\dots, 3N/4-1\}$, except for the  neighbor of $X_{\tau_{coup}(X,X')}$ in the  lower semicircle (this neighbor is either $N/2+1$ or $3N/4-1$). Define $B_{X'}$ analogously. Then again,  $\Gamma_{\tau_{coup}(X,X')}(u)\ne \Gamma'_{\tau_{coup}(X,X')}(u)$ for all $u\in B_{X'}$, a.s. and, $B_{X'}$ a.s.  contains all elements of the upper semicircle $\{3N/4+1,\dots,N-1,0,\dots,N/2-1\}$ except the neighbor of $X'_{\tau_{coup}(X,X')}$ in the upper semicircle.  Write $u_P$ for $X_{\tau_{coup}(X,X')}$ and $\tilde u_{P}$ for the second element in the equator, $\tilde u_{P}=N-u_P$. Thus $B= B_X\cup B_{X'}$ contains $V-(A_{u_P}\cup \{\tilde u_P\})$, and clearly $\Xi$ and $\Xi'$ will be coupled only after all fitnesses in $B$ are be updated.  In other words, letting 
$${\bar \sigma_B}= \inf\{t\ge0 : B \subset \cup_{s\le t} A_{X_{\tau_{coup}(X,X')+s}}\}=\inf\{t\ge 0: V-\{\tilde u_P\}\subset \cup_{s\le t} A_{X_{\tau_{coup}(X,X')+s}}\},$$ 
then we showed that 
$$ \tau_{coup}(\Xi,\Xi') \ge \tau_{coup}(X,X')+ \bar {\sigma}_B.$$ 

Suppose that $t< \tau_{coup}(X,X')+{\bar \sigma}_B$. Then there exists an element $v$ in the upper semicircle such that ${\bar v} \in B_X$ or ${v} \in B_{X'}$ with the property that, respectively,  $\Gamma_t({\bar v})= \Gamma_0 ({\bar v})$ or $\Gamma'_t(v)=\Gamma'_0(v)$. Let $f_+$ be the indicator of the event that at least one of the fitnesses in the lower semicircle is larger then $1/\delta$, and $f_-$ be the indicator that at least one of the fitnesses in the upper semicircle is less  than $\delta$, and let $f=f_+-f_-$.  Let $C_{\delta,t}$ be the event that up to time $t$ all fitnesses sampled are in $[\delta,\frac 1\delta]$. Observe that on the event $\{t<\tau_{coup}(X,X')+{\bar \sigma}_B\}\cap C_{\delta,t}$, $f(\Xi_t)-f(\Xi'_t)\ge 1$. Indeed, since all sampled fitnesses are in $[\delta,\frac{1}{\delta}]$, and all initial fitnesses for $\Xi$ are $>\frac1 \delta$, $f_-(\Xi_t)=0$, and similarly $f_+(\Xi'_t)=0$. Therefore, on this event, $f(\Xi_t)-f(\Xi'_t)= f_+(\Xi_t)+f_-(\Xi'_t)\ge 1$, because for some $v$ in the lower semicircle the fitness $\Gamma_t(v)=\Gamma_0(v)>\frac{1}{\delta}$, hence $f_+(\Xi_t)=1$, or for some vertex $v$ in the upper semicircle $\Gamma'_t(v) =\Gamma'_0(v)<\delta$, hence $f_-(\Xi'_t)=1$. 
Also, on the event $\{t\ge \tau_{coup}+{\bar \sigma}_B\}\cap C_{\delta,t}$, $f(\Xi_t)=f(\Xi'_t)=0$. Therefore, 
\begin{align*} 
2{\bar d}_t &\ge E_{\eta,\eta'} [f (\Xi_t)-f (\Xi'_t)]\\
& \ge E_{\eta,\eta'}[f(\Xi_t)-f(\Xi'_t),C_{\delta,t}]-2 P(C_{\delta,t}^c)\\
& = E_{\eta,\eta'}[f(\Xi_t)-f(\Xi'_t),C_{\delta,t}\cap \{t<\tau_{coup(X,X')}+{\bar \sigma }_B\}]- 2 P(C_{\delta,t}^c)\\
& \ge P(\tau_{coup(X,X')}+{\bar \sigma}_B>t)-3P(C_{\delta,t}^c).
	\end{align*}
The distribution of $\tau_{coup}(X,X')+{\bar \sigma}_B$ is independent of the choice of $\delta$, as it only depends on the path of the random walk $X$. Furthermore, since $G$ is vertex transitive, $\tau_{coup}(X,X')$ and ${\bar \sigma}_B$ are independent. Since $P(C_{\delta,t})\to 1$ as $\delta\to 0$, we conclude with 
$$ {\bar d}_t \ge \frac 12 P(\tau_{coup}(X,X')+{\bar \sigma}_B>t).$$	

As for the upper bounds, a simple modification of the arguments in \cite{BPT} (and specifically: the exit time from an interval is dominated by the exit time from its ``middle point", as shown in the Appendix there), shows that 
\begin{eqnarray*}
\sup_{v,u}P({\tau_{coup(X,X')}}+\bar{\sigma}>t|(X_0,X'_0)=(v,u))&=&\sup_{v,\bar{v}}P({\tau_{coup(X,X')}}+\bar{\sigma}>t|(X_0,X'_0)=(v,\bar{v}))\\
&=&P({\tau_{coup(X,X')}}+\bar{\sigma}>t|(X_0,X'_0)=(0,N/2)).
\end{eqnarray*}
	Thus, we have obtained the following bound: 
	\begin{prop}
		Consider the Local Bak-Sneppen on the $N$-cycle, with $N\in 4 {\mathbb N}$.  Let $X$ be a random walk on the $N$-cycle, starting from $0$, and let   $\tau_1,\tau_2$  and $\tau_3$ be  independent random variables with the following distributions:  
		 \begin{itemize}
		 	\item $ \tau_1\overset{\mbox{dist}}{=}\inf\{t\ge 0: |N/2 - X_t|\le N/4\}$
		 	 (the exit time from the discrete interval symmetric about $0$ containing exactly $\frac{N}{2}-1$ elements). 
		 	\item
		 	$\tau_2\overset{\mbox{dist}}{=}\inf\{t\ge 0:\cup_{s\le t}A_{X_s}=\{0,\dots,N-1\}-\{N/2\}\}.$ 
		 	\item $\tau_3\overset{\mbox{dist}}{=}\inf\{t\ge 0:\cup_{s\le t}A_{X_s}=\{0,\dots,N-1\}\}.$ 
		 \end{itemize}
		 Then 
		 	$$\frac 12P(\tau_1 + \tau_2 >t) \le {\bar d}_t \le P(\tau_1+\tau_3>t) .$$
	\end{prop}

\subsection{The Stationary Distribution}
We start with some additional notation. If $Y_1,\dots,Y_n$ are IID $\mbox{Exp}(1)$, then the minimum has distribution $\mbox{Exp}(n)$. Also, write $\mbox{Exp}^+(n)$ for the distribution of $Y_1$ conditioned that $Y_1$ is not the minimum of $(Y_1,\dots,Y_n)$. A straightforward calculation shows that if  $Y \sim \mbox{Exp}^+(n)$, then $Y$ has density 
$$\rho_n (t) = \frac{n}{n-1} e^{-t}(1-e^{-(n-1)t}),\,\,\,\,\,t>0.$$ 
Our main result is a description of the stationary distribution for the Local Bak-Sneppen model. 
\begin{theo} 
\label{th:main} Let $X_0$ be a $\mu$-distributed random variable, and for each $u \in V$, let $Z^u=(Z^u_t:t\in\Z_+)$ be a random walk on $V$, starting at $u$, independent of $X_0$. Next for each $u,v \in V$, let 
$$ \tau_{u,v}=\inf \{t\in \Z_+:Z^u\in A_v\},$$ 
and partition $V$ into the sets 
$$ V_i = \{v \in V: \tau_{X_0,v}=i\}.$$ 
Conditioned on $X_0$ and $Z$, define a random vector $\Gamma_0$ indexed by elements of $V$ according to the following rules: 
\begin{enumerate} 
\item Given the $V_i's$, the random vectors $((\Gamma_0(v):v \in V_i):V_i \ne \emptyset)$ are independent. 
\item For each nonempty $V_i$, the random variables  $(\Gamma_0(v):v \in V_i)$ are IID with 
\begin{enumerate} 
\item $\mbox{Exp}(1)$-distribution if $i=0$; and 
\item $\mbox{Exp}^+(|A_{Z^{X_0}_i}|)$ otherwise. 
\end{enumerate} 
\end{enumerate} 
Then $\Xi_0=(X_0,\Gamma_0)$ is distributed according to the stationary distribution for $\Xi$. 
\end{theo} 
\begin{proof}
In order to prove the theorem, it is convenient to taking time backwards. More precisely, we let $X_{-j} = Z^{X_0}_j$. Let ${\bar Z}_0 =X_1$ and let ${\bar Z}_t =Z^{X_0}_{t-1}$ for all $t \ge 1$. By reversibility, $\bar Z$ is a random walk on $G$ starting from initial distribution $\mu$ at time $0$. Let $\bar{\tau}_v = \inf \{ t \in \Z_+:{\bar Z}_t \in A_v\}$, and for $i\in \Z_+$, let ${\bar V}_i =\{v\in V: \bar{\tau}_v =i\}$. We observe that ${\bar V}_0= A_{X_1}$, and that by construction, 
$\bar{V}_i = V_{i-1} - {\bar V}_0$. Furthermore, the distribution of $(X_1,{\bar Z})$ and of $(X_0,Z^{X_0})$ is the same. Let $(f_v:v \in V)$ be bounded real-valued  continuous functions on ${\mathbb R}$. We need to show that 
$$ E [ \prod_{v\in V} f_v (\Gamma_1(v))]= E [ \prod_{v \in V} f_v (\Gamma_0(v))].$$ 
By construction, this is the same as proving 
$$ E [ \prod_{v\in V} f_v (\Gamma_1(v))]=E\left [  \prod_{w \in { V}_0} E [f_w (\mbox{Exp}(1))]  \times \prod_{i \ge 1}\prod_{w \in { V}_i} E [ f_w (\mbox{Exp}^+(|A_{{Z}_i}|))]|Z^{X_0}\right].$$ 
We do this by decomposing the left-hand side according to the values taken by $f_w(\Gamma_0(w))$ on the sets $\{\bar V_i, i\geq 0\}$. We have
\begin{equation}
\label{eq:start_from_this} 
  E [ \prod_{v\in V} f_v (\Gamma_1(v))]  =E [ \prod_i \prod_{v \in {\bar V}_i} f_v (\Gamma_1(v)]. 
  \end{equation}
Since, according to the process rules, $f_w(\Gamma_1(w))$ coincides with $f_w(\Gamma_0(w))$ on $V_{i-1} - A_v$, $i\geq 1$, we obtain
 \begin{equation*}
 E [ \prod_i \prod_{v \in {\bar V}_i} f_v (\Gamma_1(v)] = \sum_{u,v}E [ \prod_{w \in A_{v}} f_w (\Gamma_1(w))\times \prod_{i\ge 1}  \prod_{w \in V_{i-1} - A_v} f_w (\Gamma_0(w)); X_1 = v, X_0=u ] .
  \end{equation*}
  Now, on the event $\{X_1=v\}$, $\Gamma_1(w)$ is Exp(1) for all $w\in A_v$. Hence, the equation above is equal to
  \begin{align*}
  &\sum_{u,v} \left(\prod_{w \in A_v} E [f_w (\mbox{Exp}(1))]\times E [ E[ \prod_{i\ge 1}  \prod_{w \in V_{i-1} - A_v} f_w (\Gamma_0(w)); X_1 = v, X_0=u | Z^{X_0}] ]\right)\\
  \nonumber
  & =\sum_{v\in V}\left( \prod_{w \in A_v} E [f_w (\mbox{Exp}(1))] \times\sum_{u\in A_v}  E [ E[ \prod_{i\ge 1}  \prod_{w \in V_{i-1} - A_v} f_w (\Gamma_0(w)); X_1 = v, X_0=u | Z^{X_0}] ]\right).
\end{align*}  
We now handle the conditional expectation above. Note that 
\begin{align*} 
&E[ \prod_{i\ge 1}  \prod_{w \in V_{i-1} - A_v} f_w (\Gamma_0(w)); X_1 = v, X_0=u | Z^{X_0}]\\ 
&= E [ \prod_{i\ge 1} \prod_{w \in A_{Z^{u}_{i-1}-A_v}} f_w(\Gamma_0(w));\Gamma_0(v)=\min_{w\in A_u}\Gamma_0(w)|X_0=u,Z^{X_0}]P(X_0=u).
\end{align*}
Observe that, given $X_0$ and $Z$, the fitnesses are independent random vectors. This implies that the above expression equals 
$$
 \mu(u) E [ \prod_{w \in A_u - A_v} f_w(U_w); U_v= \min_{w \in A_u} U_w] \times E[\prod_{i\ge 2 } \prod_{w \in A_{Z^u_{i-1} -A_v}}  f_w(\mbox{Exp}^+(A_{Z^u_{i-1}}))],
$$ 
where $(U_w:w \in V)$ are IID $\mbox{Exp}(1)$. 
But, 
 \begin{align*} 
 E [ \prod_{w \in A_u - A_v} f_w(U_w); U_v= \min_{w \in A_u} U_w] &=E [  \prod_{w \in A_u - A_v} f_w(U_w)| U_v= \min_{w \in A_u} U_w] \frac{1}{|A_u|}\\
 & = \frac{1}{|A_u|}\prod_{w \in A_u - A_v} E [ f_w (\mbox{Exp}^+(|A_u|))].
 \end{align*} 
 Putting all together we obtain 
 \begin{align*} 
 &E[ \prod_{i\ge 1}  \prod_{w \in V_{i-1} - A_v} f_w (\Gamma_0(w)); X_1 = v, X_0=u | Z^{X_0}] \\
 & \quad = \frac{ \mu(u) }{|A_u|}  \prod_{w \in A_u - A_v} E [ f_w (\mbox{Exp}^+(|A_u|))] \times \prod_{i\ge 2} \prod_{w \in A_{Z^u_{i-1} -A_v}}E [ f_w(\mbox{Exp}^+(A_{Z^u_{i-1}}))].
 \end{align*} 
Now we make use of the time reversed random walk $\bar Z$. Since the distribution of $X_0$ given $X_1=v$ is equal to 
$$ P(X_0=u  |X_1= v) = P(X_1 = v|X_0=u) P(X_0=u)/P(X_1=v)= \frac{\mu(u)}{|A_u|} \frac{1}{\mu(v)},$$ 
 summing over $u\in A_v$,  we obtain 
  \begin{align*} 
 &\sum_{u\in A_v} E[ \prod_{i\ge 1}  \prod_{w \in V_{i-1} - A_v} f_w (\Gamma_0(w)); X_1 = v, X_0=u | Z^u]  \\
 & \quad = \mu(v)  \prod_{w \in \bar{V}_1} E [ f_w (\mbox{Exp}^+(|A_{{\bar Z}_1}|))|{\bar Z}_0=v] \times \prod_{i\ge 2} \prod_{w \in A_{{\bar V}_i}}E [ f_w(\mbox{Exp}^+(|A_{\bar{Z}_i}|))|{\bar Z}_0=v]\\
 &\quad = \mu(v)\prod_{i\ge 1}  \prod_{w \in {\bar V}_i} E [ f_w(\mbox{Exp}^+(|A_{{\bar Z}_i}|))|{\bar Z}_0=v].
 \end{align*}
 Plugging this into the righthand side of Equation \eqref{eq:start_from_this}, we obtain 
 $$ E [\prod_{v \in V} f_v (\Gamma_1(v)) ] =  E\left [  \prod_{w \in {\bar V}_0} E [f_w (\mbox{Exp}(1))]  \times \prod_{i \ge 1}\prod_{w \in {\bar V}_i} E [ f_w (\mbox{Exp}^+(|A_{\bar{Z}_i}| ))]
 |{\bar Z}^{X_1}\right].$$ 
 Since the distribution of ${\bar Z}$ coincides with the distribution of $Z,$ the result follows.
\end{proof}
Consider a graph $G$ of constant degree $d\ge2$.  Then  $\mu$ is uniform. Also,  since  $|A_v|=d+1$ for all $v\in V$,  the conditional distribution of  $(\Gamma_0(v):v\not \in A_{X_0})$ on  $X_0$ and $Z$,  is IID $\mbox{Exp}^+(d+1)$. In other words,  
 \begin{prop}
 \label{pr:cycle}
 Suppose that $G$ is of constant degree $d\ge 2$. That is $|A_v|=d+1$ for all $v\in V$. Let $X_0$ be uniformly distributed on $V$, $(U_v:v\in V)$ IID $\mbox{Exp}(1)$ and $(Z_v:v \in V)$ IID $\mbox{Exp}^+(d+1)$, both independent of $X_0$ and each other. Set 
 $$\Gamma_0(v) = \begin{cases} U_v & v \in A_{X_0} \\ 
 Z_v & v \not \in A_{X_0}.\end{cases}$$ 
 Then the distribution of  $\Xi_0=(X_0,\Gamma_0)$ is stationary for the local Bak-Sneppen on $G$. 
 \end{prop} 
\begin{cor}
\label{cor:fitness}
If $G$ is of constant degree $d$, then under the stationary distribution 
\begin{enumerate} 
\item The fitness distribution at every site is the convex combination $(1 - \frac{d+1}{|V|}) \mbox{Exp}^+(d+1) + \frac{d+1} {|V|}  \mbox{Exp}(1)$. 
\item Suppose that $u,v \in V$ are such that $u\not\in A_v$ (equivalently $v\not \in A_u$).  Then the joint distribution of $(\Gamma_0(u),\Gamma_0(v))$ is given by 
\end{enumerate}
$$\frac{d+1}{|V|} \mbox{Exp}(1)\otimes \mbox{Exp}^+(d+1)+ \frac{d+1}{|V|}\mbox{Exp}^+(d+1)\otimes \mbox{Exp}(1)$$ $$+  (1 - \frac{2(d+1)}{|V|}) \mbox{Exp}^+(d+1) \otimes \mbox{Exp}^+(d+1)$$ 
 
\end{cor}

Finally, we obtain a formula for the fitness density in a general graph.  To simplify notation, define 
$$ \sigma_v = \inf\{t\in \Z_+:X_t \in A_v\}.$$
We also write $P_u$ for the distribution of the random walk $X$, starting from $u\in V$. 
If $X_0=u$, then the distribution of $(X_{\sigma_v},\sigma_v)$ coincides with that of $(Z^v_{\tau_{u,v}},\tau_{u,v})$ defined above. Also for $v\in V$, let 
$$ \partial A_v = \{z\in A_v:A_z \cap A_v^c\ne\emptyset\}.$$ 
We have the following formula for the stationary density at vertex $v$:
\begin{prop}
Let $v\in V$. The stationary distribution for $\Gamma_0(v)$ is equal to 
$$ \frac{ \sum_{u\in A_v} |A_u|}{|S_G|} \mbox{Exp}(1)+\sum_{u\not \in A_v} \frac{|A_u|}{S_G} \sum_{z \in \partial A_v} P_u (X_{\sigma_v} =z)\mbox{Exp}^+(|A_z|).$$  
\end{prop} 
\begin{proof}
Suppose that $\Xi_0$ has the stationary distribution as given in Theorem \ref{th:main}. 
We clearly have 
$$  E [f (\Gamma_0(v)) ] =E [ E [ f(\Gamma_0(v) ) |X_0,Z] ].$$ 
Observe: 
$$E [ f (\Gamma_0(v));X_0\in A_v |X_0,Z]  = E [f(\mbox{Exp}(1))]{\mathbb 1}_{A_v}(X_0).$$
Since for $u\in A_v$, $\mu(u) = \frac{|A_u|}{S_G}$, this explains the first summand in the expression. 
It remains to evaluate 
\begin{equation}
\label{eq:leftover} E [f (\Gamma_0(v));X_0\not \in A_v|X_0,Z]=E [f (\Gamma_0(v))|X_0,Z]{\mathbb 1}_{A_v^c}(X_0).
\end{equation} 
Suppose then that $u\not \in A_v$, Then $v$ belongs to some $V_i,~i>1$, characterized by $i = \inf\{t\ge 0 : Z^u_i \in A_v\}$, and the (conditional) distribution of $\Gamma_0(v)$ is $\mbox{Exp}^+(|A_{Z^u_i}|)$. That is 
$$ E [ f(\Gamma_0(v))| X_0=u,Z^u] =E [ f(\mbox{Exp}^+(|A_{Z^u_{\tau_{u,v}}}|))|X_0=u,Z^u].$$ 
The distribution of $Z^u_{\tau_{u,v}}$ coincides with $P_u(X_{\sigma_v} \in \cdot)$. Since $u \in A_v^c$, and $X_{\sigma_v} \in A_v$, it follows that $X_{\sigma_v} \in \{z \in A_v:A_z \cap A_v^c\ne \emptyset\}=\partial A_v$. Taking expectation with respect to $Z^u$ gives 
$$ E[E [ f(\mbox{Exp}^+(|A_{Z^u_{\tau_{u,v}}}|))|X_0=u,Z^u]]= \sum_{z\in \partial A_v} P_u (X_{\tau_v}=z) E[f(\mbox{Exp}^+(|A_z|))].$$ 
Using this in Equation \eqref{eq:leftover} gives 
\begin{align*} E [ E [ f (\Gamma_0(v));X_0\not\in A_v |X_0,Z] ] &= \sum_{u\not \in A_v} \mu (u) E [E [ f (\Gamma_0(v)); |X_0=u,Z^u]]\\
& =\sum_{u\not \in A_v}\frac{|A_u|}{S_G}\sum_{z\in \partial A_v} P_u (X_{\tau_v}=z) E[f(\mbox{Exp}^+(|A_z|))].
\end{align*}
\end{proof} 
\subsection{The $\alpha$-Avalanches}
We now study another aspect of the model, namely its avalanches. Here, the definition of an avalanche is relaxed in the following way: given thresholds $b>0$ and $\alpha \in [0,1)$, we define an $\alpha$-avalanche from threshold $b$ as a portion of path of the process since the proportion of the vertices whose fitness is at least  $b$ exceed  $\alpha$ until, but not including, the next time this happens.  The case $\alpha=1$ coincides with the notion of an avalanche given in Section \ref{sec:bg}. The reason for introducing the $\alpha$-avalanches is because the local Bak-Sneppen is ``slow" in replacing low fitnesses as it follows a random walk on the graph, and so in typical conditions, the expected duration of avalanches ($\alpha=1$) from any threshold $b>0$ will tend to infinity as the graph becomes larger.  Note also that the  distribution of the duration of an   $\alpha$-avalanche depends on the initial configuration.  Yet, the ergodicity of the local Bak-Sneppen allows to replace the expectation  by the limit of time averages of the durations along the  sequence of consecutive  $\alpha$-avalanches, a limit which exists a.s. and is equal to a deterministic constant. 

Fix a  finite connected graph $G=(V,E)$ and consider the local Bak-Sneppen on $G$. For a threshold $b>0$ and time $t\in\Z_+$, let 
$$\Psi_t (b) = \frac{ \sum_{v \in V} {\bf 1}_{[b,\infty)} (\Gamma_t(v))}{|V|}$$ 
denote the proportion of vertices  with fitness $\ge b$ at time $t$. Let $T_0=0$, and continue inductively by letting 
$$ T_{n+1}  = \inf\{t>T_n:  \Psi_t (b)  \ge \alpha \}.$$ 
The sequence  $\{T_n\}_{n\in{\mathbb N} }$  consists of all times when the proportion of vertices with fitness greater or equal than $b$ is at least $\alpha$, and so,  represent the sequence of times where $\alpha$-avalanches begin, with the possible exception of an $\alpha$-avalanche starting at time $T_0=0$.  Observe that if $t>0$, then  $t=T_n$ for some $n$ if and only if $\Psi_t(b) \ge \alpha$. 
Let $N_t$ count the number of $\alpha$-avalanches from threshold $b$  completed by time $t$. Then for any state $\eta$, it follows from Corollary \ref{cor:erg} that 
$$\lim_{t\to\infty} \frac{N_t}{t} = \pi(A_{\alpha,b}),$$
in $P_\eta$-probability,  
where $A_{\alpha,b}$ is the set of all state $\gamma$ such that $\Psi_0(b) \ge \alpha$ when $X_0=\gamma$. That is, all states exhibiting at least a proportion $\alpha$ of the vertices with fitness $\ge b$.  Letting $t=T_n$, we have the following immediate corollary: 

\begin{cor}
\label{pr:ergodic_times}
Let $D(\alpha,b) = \frac{1}{\pi (A_{\alpha,b})}$. Then for any state $\eta$, 
$$\lim_{m\to\infty} \frac{T_m}{m} =D(\alpha,b),$$
in $P_\eta$-probability. 
\end{cor}



In what follows we assume that $d\ge 2$ and that   $(G_n=(V_n,E_n):n\in {\mathbb N})$ is a sequence of finite connected $d$-regular graphs, satisfying $V_1 \subseteq V_2 \subseteq \dots $ and $\lim_{n\to\infty} |V_n|=\infty$. Write $\pi_n$ for the stationary distribution for the local Bak-Sneppen on $G_n$. We have the following: 

\begin{prop}\label{limit_dist}
The limit    $\lim_{n\to\infty} \pi_n$ exists and is equal to the distribution  of IID $\mbox{
 Exp}^+(d+1)$-distributed random variables labeled by $V_\infty = \cup_{n} V_n$, and where  the convergence is in the weak topology on probability measures on the product space $[0,\infty)^{V_\infty}$. 
\end{prop}

\begin{proof}
 let $v_1,v_2,\dots,v_m \in V_\infty=\cup_n V_n$ be fixed, and let $(f_i:i\le m)$ be bounded real-valued continuous functions on $[0,\infty)$.  We run the local Bak-Sneppen  on $G_n$ from its stationary distribution $\pi_n$. Then $X_0$ is uniformly distributed on $V_n$. This gives (for some constant $K>0$),  
$$ \left | E[ \prod_{i=1}^m f_i(\Gamma_0(v_i)) ] - E [ \prod_{i=1}^m f_i (\Gamma_0(v_i));  X_0 \not \in \{v_1,\dots,v_m\} ]\right |\le$$ $$ \le KP(X_0 \in \{v_1,\dots,v_m\}) = \frac{ m}{|V_n|} \to 0,$$ as $n$ goes to infinity. 
However, 
\begin{align*}  E[  \prod_{i=1}^m f_i (\Gamma_0(v_i));  X_0 \not \in \{v_1,\dots,v_m\}]  ] &=  E[  \prod_{i=1}^m f_i (\Gamma_0(v_i)) | X_0 \not \in \{v_1,\dots,v_m\} ] (1-\frac{m}{|V_n|})\\
& = \prod_{i=1}^m E [ f_i (\mbox{Exp}^+(d+1))](1+ o(1)).
\end{align*} 
This proves convergence of finite dimensional distributions.  Finally, tightness is obtained by observing that if $V_\infty = \{1,2,\dots\}$, and since $\mbox{Exp}^+(d+1)$ stochastically dominates $\mbox{Exp}(1)$, then 
$$ P( \cap_{i=1}^\infty  \{\Gamma_0(i) \in [0,N+i]\}) \ge \prod_{i=1}^\infty \left (1- P( \mbox{Exp}^+(d+1) >N+i)\right ) \ge \prod_{i=1}^\infty (1-\frac{d+1}{d}e^{-(N+i)})$$ for all $N\in\mathbb{N}$.
Although the lefthand side depends on $\pi_n$, the righthand side does not. By dominated convergence ($N=0$), the righthand side converges to $1$ as $N\to\infty$. This completes the proof. 
\end{proof} 
Now let  $$ p_{d,b} \equiv P( \mbox{Exp}^+(d+1) \ge  b) = \int_b^\infty \rho_{d+1} (t) dt = \frac{d+1}{d} (e^{-b}-\frac{1}{d+1}e^{-(d+1)b}).$$ We have the following:
\begin{prop}\label{avalanches}  Fix  $\alpha \in (0,1)$. Let $b_c\in (0,\infty)$ be the unique solution to 
$$\alpha= p_{d,b}.$$
Then 
\begin{enumerate} 
\item If $b<b_c$, then for some $\rho_1>0$,  $|D_n(\alpha,b) -1 |\le e^{-\rho_1 n}$. 
\item if $b=b_c$, then for some $c>0$, $|D_n(\alpha,b) -2|\le  \frac{c}{\sqrt{n}}$. 
\item If $b>b_c$, then for some $\rho_2>0$, $D_n (\alpha,b) \ge e^{\rho_2 n}$.
\end{enumerate} 
\end{prop} 
\begin{proof}
Letting  $m = \lfloor \alpha |V_n|\rfloor$,  the event $A_{\alpha,b}$ coincides with the event that at least $m$ of the $|V_n|$ sites have fitness above $b$. Conditioning on $X_0$ we obtain
\begin{equation}\label{eq:A}
\pi_n(A_{\alpha,b})=\sum_{k=1}^{|V_n|}\pi_n(A_{\alpha,b}|X_0=v_k)\pi_n(X_0=v_k),
\end{equation} 
where $v_1,v_2,\dots,v_{|V_n|}$ is some fixed ordering of the elements of $V_n$.  

Given $X_0$, the event $A_{\alpha,b}$ contains the event that at least $m$ among the $|V_n|-(d+1)$ vertices,  not neighbors of $X_0$, have fitness above $b$. Furthermore, given $X_0$, it follows from Proposition \ref{pr:cycle} that, under the stationary distribution $\pi_n$,  the fitness of vertices outside the neighborhood of $X_0$ are IID $Exp^+(d+1)$ random variables.   
Now given $v\in V_n$, and $u \in V_n-A_v$, we have 
$$p_{d,b}=\pi_n( \Gamma_0 (u) \ge b | X_0 = v).$$ 
Hence, if  $I_1,I_2,\dots$ are IID $\mbox{Ber}(p_{d,b})$, and $S_n = I_1+\dots + I_n$, then from Equation (\ref{eq:A}) we obtain
$$\pi_n(A_{\alpha,b})\geq\sum_{k=1}^{|V_n|}P( S_{|V_n|-(d+1)} \ge m)\pi_n(X_0=v_k).$$
Since $G_n$ is $d$-regular, $\pi_n(X_0=v_k)$ is uniform on the vertices of $G_n$. Therefore
$$ \pi_n(A_{\alpha,b})\ge P( S_{|V_n|-(d+1)} \ge m). $$
On the other hand, again conditioning on $X_0$, the event $A_{\alpha,b}$ is contained in the event that at least $m-(d+1)$ vertices in $V_n-A_{X_0}$ have fitness above $b$, and so we have 
$$\pi_n(A_{\alpha,b}) \le P( S_{|V_n|-(d+1)}\ge m- (d+1)).$$ 
It follows from the Law of Large Numbers and from the Central Limit Theorem that 

$$\lim_{n\to\infty} \pi_n (A_{\alpha,b}) =\begin{cases} 1& \alpha <  p_{d,b},\\ \frac 12 & \alpha =p_{d,b} ,\\ 0 & \alpha > p_{d,b}.\end{cases}$$ 
Furthermore, from large deviations we know that the convergence in the first and last cases occurs at an exponential rate, while for the case $\alpha =p_{d,b}$, Berry-Essen theorem (see \cite{Be} for example) guarantees that $$|\pi_n (A_{\alpha,b}) - \frac 12 |= O(n^{-1/2}).$$ 
Finally, observe that $b\to p_{d,b}$ is a continuous decreasing function, with $p_{d,0}=1$ and $\lim_{b\to\infty} p_{d,b}=0$. Therefore for any given $\alpha\in (0,1)$, there exists a unique $b_c>0$, such that $p_{d,b_c} = \alpha$, $\alpha < p_{d,b}$ if $b<b_c$ and $\alpha > p_{d,b}$ if $b>b_c$. This completes the proof. 
\end{proof}

\section*{Appendix}
\begin{proof}[Proof of Corollary \ref{cor:unique_stat_dist}]
First we show that $\Xi$ has a unique stationary measure. Fix $\eta \in \Omega$. Then for  every event $A$, 
$(P_\eta (\Xi_t \in A):t\in \Z_+)$ is a Cauchy sequence since 
\begin{equation} 
\label{eq:cauchy} | P_\eta (\Xi_t \in A) - P_\eta(\Xi_{t+s}\in A) |=  E_\eta[\left | P_{\eta}(\Xi_t \in A )-  P_{\Xi_s}(\Xi_t\in A)|\right] \le \bar{d}(t)\le c \beta^t.
\end{equation}
As a result, for each $A$,  $\pi_\eta(A)= \lim_{t\to\infty }P_\eta(\Xi_t \in A)$ exists.  In order to show that $\pi_\eta$ is a probability measure, observe first that $\pi_\eta(\emptyset)=0$, $\pi_\eta (\Omega)=1$, and that $\pi_\eta$ is monotone with respect to inclusion.  Assume  $A_1,A_2,\dots$ is a sequence of disjoint events. Fix $\epsilon>0$, and let $t$ be large enough so that $c\beta^t < \epsilon $. Then 
$$|\pi_\eta (\cup_{j=1}^\infty  A_j) - P_\eta ( \Xi_t \in \cup_{j=1}^\infty A_j) |\le c \beta ^t< \epsilon.$$ 
There exists $N$ such that $P_\eta (\Xi_t \in \cup_{j> n} A_j ) <\epsilon$ whenever $n\ge N$, and it follows from Equation \eqref{eq:cauchy} that $P_\eta(\Xi_{t+s} \in \cup _{j>n} A_j ) < \epsilon + c \beta^t < 2\epsilon$ for all $n \ge N$ and $s\ge 0$. 
This gives
\begin{align*} 
 \pi_\eta (\cup_{j=1}^\infty A_j) &\le   P_\eta (\Xi_t \in \cup_{j=1}^\infty A_j) +\epsilon\\
  &  = \sum_{j=1}^n P_\eta (\Xi_t \in A_j) + P_{\eta} (\Xi_t \in \cup_{j>n} A_j) + \epsilon\\
  &  \le \sum_{j=1}^n P_\eta (\Xi_t \in A_j) + 3\epsilon\\
  &\underset{t\to\infty}{\to} \sum_{j=1}^n \pi_\eta (A_j)+3\epsilon\\
  &\le \sum_{j=1}^\infty \pi_\eta (A_j)+3\epsilon.
  \end{align*}
  On the other hand,  it immediately follows from Fatou's lemma that 
  $$\pi_\eta (\cup_{j=1}^\infty A_j) \ge \sum_{j=1}^\infty \pi_\eta (A_j),$$ 
  implying that  $\pi_\eta$ is indeed a probability measure. 
  
  Next we prove that $\pi_\eta$ is a stationary measure for $\Xi$. From its definition and bounded convergence,  $P_{\pi_\eta}(\Xi_s \in A )= \lim_{t\to\infty} E_{\eta} P_{\Xi_t}(\Xi_s \in A)=\lim_{t\to\infty}P_\eta (\Xi_{t+s}\in A)=\pi_\eta(A)$.   
  
  The final step is to show uniqueness.  Observe that, for  $\eta,\eta' \in \Omega$, 
   \begin{align*} 
    |\pi_\eta (A) -\pi_{\eta'} (A) | &\le |P_\eta (\Xi_t \in A) - \pi_\eta (A)| + | P_\eta (\Xi_t \in A) - P_{\eta'} (\Xi_t \in A) |\\
    & + |P_{\eta'} (\Xi_t \in A) -\pi_{\eta'}(A)|.
    \end{align*}
  Hence $|\pi_\eta (A) -\pi_{\eta'} (A) | \to  0$ when $t$ goes to infinity and  $\pi_\eta =\pi_{\eta'}$. Now if $\pi$ is a stationary measure for $\Xi$,  then for $s,t\ge 0$, we have 
 
$$ \pi(A) = P_\pi (A) = P_{\pi} [P_{\Xi_s}  (\Xi_t \in A)] \underset{t\to\infty} {\to} P_\pi [\pi_{\Xi_s} (A)],$$ 
and the result follows because $\pi_{\Xi_s}$ is independent of $\Xi_s$.
\end{proof} 
\begin{proof} [Proof of Corollary \ref{cor:erg}]
Without loss of generality, assume $\int f d\pi =0$. Let $S_n = \sum_{t=0}^{n-1}f (\Xi_t)$.  We need to prove $S_n/n\to 0$ in $P_\eta$-probability for every $\eta$. This will follow if we show   $E_\eta [S_n^2/n^2]\to 0$ as $n\to\infty$. Now  
  $$ E_{\eta} [S_n^2] = \sum_{t=0}^{n-1}E_{\eta} [f^2(\Xi_t)]+2\sum_{t=0}^{n-1}\sum_{s=1}^{n-1-t} E_{\eta }[f(\Xi_t) f(\Xi_{t+s})].$$ 
   Observe that $E_{\eta }[f(\Xi_t) f(\Xi_{t+s})]= E_{\eta}[f(\Xi_t)E_{\Xi_t}[f(\Xi_s)]]$.  By Proposition \ref{experg} and Corollary \ref{cor:d_bounds}
  	$$| E_{\eta}[f(\Xi_t)E_{\Xi_t}[f(\Xi_s)]]|\le c E_{\eta }[|f(\Xi_t)|]\beta^s=c\|f\|_\infty \beta^s.$$ 
  	Therefore,  
  	$$\sum_{t=0}^{n-1}\sum_{s=1}^{n-1-t} E_{\eta }[f(\Xi_t) f(\Xi_{t+s})]\le \sum_{t=0}^{n-1}  \sum_{s=1}^{ n-1-t} c \|f\|_\infty\beta^s \le \frac{cn}{1-\beta}\|f\|_\infty.$$
	Hence $E [ S_n^2 /n^2]=O(n^{-1})$, and the result follows. 
 \end{proof}
\section*{Acknowledgements}  Iddo Ben-Ari was partially supported by Simons Foundation, grant number 208728. Roger W. C.
Silva was partially supported by Funda\c{c}\~{a}o de Amparo \`{a} Pesquisa do Estado de Minas Gerais (FAPEMIG), grant number APQ-02743-14. The second author also thanks UCONN-University of Connecticut  where an early version of this paper was finished during his visit.

\end{document}